\documentclass[12pt]{article}

\usepackage{amsthm, amssymb, amsmath, amsfonts, mathrsfs}

\usepackage{mathtools}
\usepackage[ocgcolorlinks, colorlinks=true, pdfstartview=FitV, linkcolor=blue, citecolor=blue, urlcolor=blue, pagebackref=false]{hyperref}

\usepackage{microtype}

\usepackage{color}
\usepackage{accents}

\usepackage{bbm}

\definecolor{labelkey}{gray}{.8}
\definecolor{refkey}{gray}{.8}

\definecolor{darkred}{rgb}{0.9,0.1,0.1}
\definecolor{darkgreen}{rgb}{0,0.5,0}

\newtheorem{theorem}{Theorem}[section]
\newtheorem{lemma}[theorem]{Lemma}

\newtheorem{proposition}[theorem]{Proposition}

\theoremstyle{remark}

\renewenvironment{proof}[1][Proof]{ {\itshape \noindent {#1.}} }{$\Box$
\medskip}

\numberwithin{equation}{section}
\newcommand{\R}{\mathbb{R}}

\newcommand{\Pb}{\mathbb{P}}

\newcommand{\E}{\mathbb{E}}

\newcommand{\F}{\mathcal{F}}

\newcommand{\B}{\mathrm{B}}

\newcommand{\bbR}{\mathbb R}

\newcommand{\eps}{\varepsilon}
\def\les{\lesssim}

\newcommand{\la}{\langle}
\newcommand{\ra}{\rangle}

\newcommand{\cX}{\mathcal{X}}
\newcommand{\cY}{\mathcal{Y}}

\newcommand{\lam}{\lambda}
\newcommand{\what}{\widehat}

\begin{document}

\title{The Schr\"odinger equation with spatial white noise: the average wave function}

\author{Yu Gu\thanks{Department of Mathematics, Carnegie Mellon University, Pittsburgh, PA 15213, USA (yug2@math.cmu.edu)} \and Tomasz Komorowski\thanks{Institute of Mathematics, Polish Academy of Sciences, ul. \'{S}niadeckich 8, 00-656 Warsaw, Poland (komorow@hektor.umcs.lublin.pl)} \and Lenya Ryzhik\thanks{Department of Mathematics, Building 380, Stanford University, Stanford, CA, 94305, USA (ryzhik@stanford.edu)}}
\date{}

\maketitle

\begin{abstract}
 We prove a representation for the average wave function of the Schr\"odinger equation with a white noise potential in $d=1,2$, 
 in terms of the renormalized self-intersection local time of a Brownian motion.
\end{abstract}

\section{Introduction }

We consider the Schr\"odinger equation with a large, highly oscillatory random potential 
\begin{equation}\label{e.maineq}
i\partial_t\psi_\eps+\frac12\Delta\psi_\eps-V_\eps(x)\psi_\eps=0,  \   \ \psi_\eps(0,x)=\phi_0(x),  \  \ x\in \R^d,
\end{equation}
and the initial condition $\phi_0(x)$ that is a compactly supported  $C^\infty$
 function. The random potential is a microscopically smoothed version of a spatial white noise:
\[
V_\eps(x)=\frac{1}{\eps^{d/2}}V(\frac{x}{\eps}).
\]
Here, $V$ is a stationary, zero-mean and isotropic Gaussian random
field over a probability space $(\Omega,{\cal A},\mathbb P)$ with the expectation denoted by $\E$. If the two-point correlation
function 
\begin{equation}
\label{022502a}
R(x):=\E[V(x+y)V(y)] =\rho(|x|),\quad x,y\in\bbR^d,
\end{equation}
decays sufficiently fast, then $V_\eps(x)$ converges to  a spatial white noise $\dot W(x)$.  
%Suppose that the potential  is of the form
%\begin{equation}\label{e.vgauss}
% V(x)=\int_{\R^d} q(x-y)\dot{W}(y)dy,
% \end{equation}
%where $ q(x):=q_1(x)$ with
%\begin{equation}
%\label{q-eps}
%q_\eps(x):=\frac{1}{(2\pi \eps)^{d/2}}e^{-|x|^2/(2\eps)},\quad \eps>0,
%\end{equation}
%and $\dot{W}(y)$ is a spatial white noise on 
% $(\Omega,\mathcal{A},\Pb)$, i.e., it is a 
% ${\cal S}'(\bbR)$-valued Gaussian field with the covariance 
% \[
% \E[\langle\varphi_1,\dot{W}\rangle\langle
% \varphi_2,\dot{W}\rangle]=\int_{\bbR^d}\varphi_1(x)\varphi_2(x)dx,\quad
% \varphi_1,\varphi_2\in {\cal S}(\bbR).
% \]
%By the scaling property of $\dot{W}(x)$, there exists another spatial white noise $W'$ such that 
%\[
%V_\eps(x) \stackrel{\text{law}}{=}  \int_{\R^d} q_{\eps^2}(x-y)W'(y)dy,
%\]
%and the r.h.s. scales to $W'$ as $\eps\to0$.

%Our interest is in $d=1,2$, as in the higher dimensions the limiting equation with the white noise spatial potential is super-critical.
In $d=2$, this problem was analyzed on a torus $\mathbb{T}^2$ in
\cite{debussche2016schr} and the whole space $\R^2$ in
\cite{debussche2017schr}. The solution of (\ref{e.maineq}) acquires a large phase
by $t\sim O(1)$, and the main result of \cite{debussche2016schr} is that the adjusted solution
\[
\phi_\eps(t,x)=\psi_\eps(t,x)e^{-iC_\eps t},
\]
that satisfies
\begin{equation}\label{e.maineq1}
i\partial_t\phi_\eps+\frac12\Delta\phi_\eps-(V_\eps(x)+C_\eps)\phi_\eps=0,  \   \ \phi_\eps(0,x)=\phi_0(x),  \  \ x\in \R^d,
\end{equation}
with $C_\eps\sim\log\eps^{-1}$,
converges to the solution of the stochastic PDE that can be formally written as
\begin{equation}\label{e.spde}
i\partial_t \phi_{\mathrm{spde}}+\frac12\Delta\phi_{\mathrm{spde}}-\dot{W}(x)\cdot\phi_{\mathrm{spde}}=0.
\end{equation}
The approach is based on a change of variable used in
\cite{hairer2015simple}, together with the mass and energy
conservations, and also applies to nonlinear equations. By
analyzing the Anderson Hamiltonian 
\[
-\frac12\Delta+V_\eps(x)+C_\eps
\]
with
the paracontrolled calculus, a spectral theory has been established in
\cite{allez2015continuous}, which also gives a meaning to the solution
to \eqref{e.spde} on $\mathbb{T}^2$. 

When $d=1$, no renormalization is
needed and $C_\eps=0$. It has been proved in \cite{zhang2012convergence} that
the solution $\phi_\eps$ of \eqref{e.maineq1}  converges to a solution to
\eqref{e.spde}, defined as an infinite series of iterated Stratonovich integrals.

%When the noise is too irregular (that may happen in high dimensions), the product of $V_\eps\phi_\eps$ in \eqref{e.maineq} may become ill-posed as $\eps\to0$ and a large renormalization constant $C_\eps$
%needs to be added to the potential in order to make sense of the limit. The equation then becomes
%\begin{equation}\label{e.maineq1}
%i\partial_t\phi_\eps+\frac12\Delta\phi_\eps-(V_\eps(x)+C_\eps)\phi_\eps=0,  \   \ \phi_\eps(0,x)=\phi_0(x),  \  \ x\in \R^d.
%\end{equation}
% We expect the limit of $\phi_\eps$ to solve an stochastic partial
% differential equation (for brevity we shall refer to this case as the SPDE regime) formally written as 
%\begin{equation}\label{e.spde}
%i\partial_t \phi_{\mathrm{spde}}+\frac12\Delta\phi_{\mathrm{spde}}-W'(x)\cdot\phi_{\mathrm{spde}}=0.
%\end{equation}
%
%In $d=2$,  the equation \eqref{e.spde} has been analyzed in
%\cite{debussche2016schr} on a torus $\mathbb{T}^2$:  with $C_\eps\sim
%\log\eps^{-1}$, $\phi_\eps$ converges to the solution to some limiting
%equation.  

Unfortunately, the information  on the limit  from the above considerations
is rather implicit. 
Our goal here is to understand some of the properties of the solution to~\eqref{e.maineq1}, in a more direct way. In particular,
we establish a  representation of 
$\lim_{\eps\to 0}\E[\widehat{\phi}_\eps]$ in~$d=1,2$, see Theorem
\ref{t.mainth} below. Here, and in what follows $\what f$ denotes the
Fourier transform of a  function $f$: %\bbR^d\to\mathbb C$: 
\[
\what{f}(\xi)=\int_{\bbR^d} f(x)e^{-i\xi\cdot x}dx.
\]

%\subsection{Isotropic potentials with Schoenberg class covariance}

\subsubsection*{The self-intersection local time of Brownian motion}
%\label{s.deflocal}

The representation for $\E[\what\phi_\eps]$ we are pursuing relies on the self-intersection local time of Brownian motion. For the convenience of the reader, we provide a brief introduction here. Let~$\{B_t,t\geq 0\}$ be a standard
$d-$dimensional Brownian motion starting from the origin, defined on a probability space $\Sigma$, with the respective expectation $\E_\B$.
 
In $d=1$,
one can show \cite{chen} that for any $f\in L^1(\bbR)$ with $\int_{\bbR}   f(x)dx=1$
and $t>0$, the following limit exists and represents the intersection time of the Brownian motion:
\begin{equation}
\label{011804}
 \beta([0,t]_<^2):=\lim_{\eps\to0}\frac{1}{\eps}\int_0^t ds\int_0^s
  f\left(\frac{B_s-B_u}{\eps}\right)du=\frac{1}{2}\int_{\bbR}l^2(t,x)dx,\quad
  \mbox{a.s. and in $L^1(\Sigma)$},
 \end{equation}
where $l(t,x)$ is the local time of $\{B_t,t\geq 0\}$. Here, given a subset $A\subset[0,+\infty)$, we denote 
by~$A^2_{<}:=\{(s,t)\in A^2:\,s<t\}$.
On the formal level, we can think of  $\beta([0,t]_<^2)$ as 
\[
\beta([0,t]_<^2)=\int_0^t \int_0^s \delta(B_s-B_u) duds.
\] 

The direct analogue of the  self-intersection time 
%that corresponds to equality  \eqref{011804} for a Brownian motion 
in dimensions $d\ge 2$ becomes infinite, and a
suitable renormalization is needed  to recover a non-trivial
object. The {renormalized self-intersection local time of a planar Brownian
motion} $\gamma([0,t]_<^2)$  formally corresponds to
\begin{equation}\label{jun110}
\gamma([0,t]_<^2)=\int_0^t \int_0^s (\delta(B_s-B_u)-\E_\B[\delta(B_s-B_u)]) duds.
\end{equation}
To make sense of (\ref{jun110}) in $d=2$, one defines the renormalized
self-intersection local time  as
\begin{equation}
\label{032204}
\gamma([0,t]_<^2):=\lim_{\eps\to0}\int_0^t \int_0^s \left\{ \vphantom{\int_0^1}q_\eps(B_s-B_u)-\E_\B[q_\eps(B_s-B_u)]\right\} duds.
\end{equation}
The limit exists for any $t>0$~\cite{le1992some,varadhan1969appendix,yor1986precisions}.
Here, we denote 
\begin{equation}
\label{q-eps}
q_\eps(x):=\frac{1}{(2\pi \eps)^{d/2}}e^{-|x|^2/(2\eps)},\quad \eps>0.
\end{equation}
%The quantity $\gamma([0,t]_<^2)$ is called time, 
We refer to 
\cite[Section VIII.4]{le1992some} for a detailed construction.

In $d=1$, we simply let
\begin{equation}
\label{011904}
\gamma([0,t]_<^2):=\beta([0,t]_<^2)-\E_{\B}[\beta([0,t]_<^2)]=\frac{1}{2}\left\{\int_{\bbR}l^2(t,x)dx-\frac{8t^{3/2}}{3\sqrt{2\pi}}\right\}.
\end{equation}
%\gucomment{my calculation gives $\frac{8t^\frac32}{3\sqrt{2\pi}}$}

\subsection*{The main result}

We will assume that the covariance function of the Gaussian random field $V(x)$ has the
form~(\ref{022502a}), with a function $\rho(y)$ of
%\begin{equation}
%\label{022502a}
%R(x):=\E[V(x+y)V(y)] =\rho(|x|),\quad x,y\in\bbR^d,
%\end{equation}
% where
%$\rho:\R\to\R$  is of the
the Schoenberg class~\cite{schoenberg}:
\begin{equation}
\label{012502a}
\rho(y)=\int_0^{+\infty}\exp\left\{-(\lambda y)^2/2 \right\}\mu(d\lambda),\quad y\in \R,
\end{equation}
for some finite Borel measure $\mu$ on $[0,+\infty)$.
To ensure that $V_\eps(x)$ scales to a spatial white noise with a finite variance, we assume that
\begin{equation}
\label{021904}
\bar R_d:=\int_{\bbR^d}R(x)dx=(2\pi)^{d/2}\int_0^{+\infty}\frac{\mu(d\lambda)}{\lam^d}<+\infty.
\end{equation}
%With appropriate choices of $\mu(d\lambda)$, the covariance function $R(x)$ can be merely integrable. 
To simplify some considerations, we further require that
\begin{equation}\label{e.r2}
\int_0^{+\infty} \frac{|\log\lambda|}{\lambda^d}\mu(d\lambda)<+\infty,
\end{equation}
and define 
\begin{equation}
\bar R_2'=2\pi\int_0^{+\infty}  \frac{\log\lambda}{\lambda^d}\mu(d\lambda).
\end{equation}
The constraint \eqref{e.r2} on $\mu(d\lambda)$ near the origin can be relaxed, as discussed at the end of the proof of Lemma~\ref{l.rho}
below but we are not striving for the sharpest assumptions here. 

%To simplify the presentation of the main theorem, we also assume
%\begin{equation}\label{e.extraass}
%\int_0^{+\infty} \frac{|\log \lambda|}{\lambda^d} \mu(d\lambda)<+\infty
%\end{equation}
%and define
%\begin{equation}
%\label{021904-2}
%\bar R_2':=4\pi\int_0^{+\infty}\frac{\log \lambda}{\lambda^d} \mu(d\lambda).
%\end{equation}
%The assumption \eqref{e.extraass} can be removed if we slightly change the renormalizaton constant, see Remark~\ref{r.longrange}.
%
%The power energy spectrum  of the potential, which is  the Fourier
%transform of its covariance function,
%is then of the form
%$\what R(p)={\cal E}(|p|)$, with
%\begin{equation}
%\label{schoen2abis}
%{\cal E}(k)=\int_0^{+\infty}e^{-(\lam k)^2/2}\nu(d\lam),\quad k>0,
%\end{equation}
%and  $\nu(\cdot)$ is a Borel measure on $(0,+\infty)$ given by 
%\begin{equation}
%\label{nu}
%\nu(d\lam):=(2\pi)^{d/2}\lam^d\mu S^{-1}(d\lam).
%\end{equation}
% Here
%$S:(0,+\infty)\to(0,+\infty)$ is  $S(\lam):=\lam^{-1}$. Condition
%\eqref{021904} is equivalent with
%\begin{equation}
%\label{021904a}
%\int_0^{+\infty}\nu(d\lambda)<+\infty.
%\end{equation}
%
%The potential in \eqref{e.vgauss} corresponds to the case where $\mu(d\lambda)=(4\pi)^{-d/2}\delta(\lambda-\frac{1}{\sqrt{2}})$ and $\rho(y)=(4\pi)^{-d/2}e^{-y^2/4}$.
%

Define the deterministic function
%\gucomment{the constants are different}
\begin{equation}\label{e.defrho0}
\rho_d(t):=\left\{\begin{array}{ll}
-\bar R_1 \dfrac{(2t)^\frac32}{3\sqrt{i\pi}},& d=1,\\
&\\
 \bar R_2\left[ \dfrac{it}{2\pi}\log\left(
                    \dfrac{t}{e}\right)-\dfrac{t}{4}\right]+\bar R_2' \dfrac{it}{\pi}, & d=2,
  \end{array}
  \right.
\end{equation}
and the renormalization constant
\begin{equation}\label{e.defrho1}
C_\eps:=\left\{\begin{array}{ll}
0, & d=1, \\
&\\
 \dfrac{\bar R_2}{\pi} \log\eps^{-1},  & d=2.
  \end{array}
  \right.
\end{equation}
The following result is the main objective of this paper.
\begin{theorem}\label{t.mainth}
Suppose that $d=1,2$ and $\phi_\eps$ is the solution to \eqref{e.maineq1} with
$C_\eps$ given by \eqref{e.defrho1}. Then,
there exists $t_0\in (0,+\infty)$ such that for $t\in [0,t_0],\xi\in\R^d$, 
we have 
\begin{equation}\label{e.mainresult}
\lim_{\eps\to0}\E[\widehat{\phi}_\eps(t,\xi)]=\what{\phi}_0(\xi)e^{\rho_d(t)}
\E_\B\left[ \exp\left\{i^{3/2}\xi\cdot B_t-i^{3d/2} \bar R_d\gamma([0,t]_<^2)\right\}\right].
\end{equation}
%where $\rho_d(t)$,  is  given in \eqref{e.defrho0}.
\end{theorem}

Without the random potential, the solution to the free Schr\"odinger equation  can be written in the Fourier domain as 
\[
\what{\phi}_0(\xi) \exp\Big\{-\frac{i}{2}|\xi|^2t\Big\}=\what{\phi}_0(\xi) \E_\B[\exp(i^{3/2}\xi\cdot B_t)],
\]
so Theorem~\ref{t.mainth} shows that the effect of the white noise potential is manifested by 
the term~$ \bar R_d\gamma([0,t]_<^2)$ in \eqref{e.mainresult}. 

We briefly comment on our choice of the covariance function to be in the Schoenberg class. First, since we are interested in the limiting SPDE, the way by which the noise is regularized essentially does not affect the expression in \eqref{e.mainresult}. Secondly, most of the existing results on singular SPDEs considered random fields that decorrelate sufficiently fast or even with a finite range of correlation. In our case, with appropriate choices of $\mu(d\lambda)$ in \eqref{012502a}, the covariance function $R(x)$ can be merely integrable, which is necessary to guarantee the finiteness of $\bar R_d$. Lastly, the Schoenberg class also helps avoid several technical issues, e.g. in the proof of Propositions~\ref{p.fk} and \ref{p.weakcon}.

\subsubsection*{The stochastic and homogenization regimes}

Equation \eqref{e.maineq} is written in terms of the macroscopic variables. 
If we start from the microscopic dynamics -- the Schr\"odinger equation with a potential of a size~$\delta>0$ and a low frequency initial condition,
varying on a spatial scale $l_{\mathrm{in}}\sim \eps^{-1}\gg 1$,
\begin{equation}\label{e.unre}
i\partial_t \phi+\frac12\Delta \phi-\delta V(x)\phi=0,  \  \ \phi(0,x)=\phi_0(\eps x)
\end{equation}
then $\psi_\eps(t,x):=\phi({t}/{\eps^2},{x}/{\eps})$ solves
\eqref{e.maineq} provided $\eps=\delta^{1/(2-d/2)}$. In particular, in $d=2$, we need to choose $\eps=\delta$ to be in the ``white-noise" scaling
of (\ref{e.maineq}). In other words, the white-noise scaling in $d=2$ is equivalent to the weak coupling scaling with a low frequency initial condition. 

%When $\delta=\eps$, we are in the classical weak-coupling regime.
% \[
% \delta=\left|\begin{array}{ll}
% \eps^\frac32 & d=1, \\
% \eps  & d=2.
%   \end{array}
%   \right.
% \]

%In the weak-coupling regime, it is well-known that if the initial condition is of high frequency, e.g.,  $\phi(0,x)=\phi_0(x)$ in \eqref{e.unre},  then, we observe the
%kinetic behavior of the (rescaled) energy density
%$\eps^{-2d}|\phi(\eps^{-2}t,\eps^{-2}x)|^2$%\gucomment{should it be
%                                %$\eps^{-2d}$ instead of $\eps^{-d}$?}
% in $d\geq
%2$, where the wave energy density propagates
%like the density function of an ensemble of particles traveling, 
%scattering and non-interacting with each other
%\cite{spohn1977derivation,erdHos2000linear,chen2016weak}. 
%
It has been shown in \cite{chen2016weak,zhang2014homogenization} that in $d\geq3$, for the low frequency initial data
$\phi(0,x)=\phi_0(\eps x)$, the diffusively rescaled wave function $\phi_\eps(t,x)=
\phi(\eps^{-2}t,\eps^{-1}x)$ converges to a homogenized limit: the solution has a deterministic limit, 
and we only observe a phase shift of the wave function in the limit, by a factor proportional to
\begin{equation}\label{jun102}
V_{\mathrm{eff}}=\int_{\bbR^d} \frac{\what{R}(p)dp}{|p|^{2}}.
\end{equation}
The integral in \eqref{jun102} blows up in $d=2$ due to the singularity at the origin, and the
role of the large constant $C_\eps$ appearing in \eqref{e.maineq1} is
to compensate for this divergence, so that we can obtain a  non-trivial  limit, which is now random, unlike in $d\ge 3$. 
One may ask if there is a shorter time scale~$T_\eps$, on which the solution of (\ref{e.unre}) is affected in a non-trivial way but is still deterministic
in~$d=2$.
The answer is  given by the following theorem:  $T_\eps=\eps^{-2}$, 
with~$\delta=\eps|\log\eps|^{-1/2}$.

%Our focus here on the SPDE regime in $d=2$ falls into the classical weak-coupling regime. Due to a low frequency condition $\phi(0,x)=\phi_0(\eps x)$ in \eqref{e.unre}, we do not observe any kinetic behavior. The difference from the homogenization regime in $d\geq 3$ is
%that the ``homogenized'' constant in
%\cite{zhang2014homogenization,chen2016weak}, which takes the form $\int \frac{\what{R}(p)}{|p|^{-2}}dp$, blows up in $d=2$. The
%role of the large constant $C_\eps$ appearing in \eqref{e.maineq1} is
%to compensate for this divergence, so that we can obtain a remaining non-trivial random limit.
%A byproduct of the proof of Theorem~\ref{t.mainth} is a homogenization result for the
%solution to \eqref{e.unre} if we choose $\delta=\eps|\log\eps|^{-1/2}$.
%
\begin{theorem}\label{t.homo}
Consider 
\begin{equation}
\label{012104}
i\partial_t \phi_\eps+\frac12\Delta\phi_\eps- \frac{1}{\eps|\log\eps|^\frac12}V(\frac{x}{\eps})\phi_\eps=0,  \  \  \phi_\eps(0,x)=\phi_0(x), \  \ x \in\R^2,
\end{equation} 
and 
\begin{equation}
\label{012104h}
i\partial_t \phi_{\hom}+\frac12\Delta\phi_{\hom}+ \frac{\bar R_2}{\pi}\phi_{\hom}=0,  \  \  \phi_{\hom}(0,x)=\phi_0(x), \  \ x \in\R^2,
\end{equation}
with $\phi_0\in L^2(\R^2)$. Then, for any $t>0$, we have%there exists $t_0\in (0,\infty)$ such that for all $t\in [0,t_0]$, we have 
\begin{equation}
\lim_{\eps\to0}\int_{\R^2} \E[|\what{\phi}_\eps(t,\xi)-\what{\phi}_{\hom}(t,\xi)|^2] d\xi=0.
\end{equation}

\end{theorem}

%\subsubsection*{Higher dimensions}
\subsubsection*{The non-diagrammatic approach}

The standard approach to the random Schr\"odinger equation in the weak coupling regime  
is through a diagram expansion: the solution to \eqref{e.unre} is written in the mild formulation
\begin{equation}\label{e.duhamel}
\what{\phi}(t,\xi)=\what{\phi}(0,\xi)e^{-\frac{i}{2}|\xi|^2t} +\delta\int_0^t e^{-\frac{i}{2}|\xi|^2(t-s)}\left(\int_{\bbR^d}\frac{\what{V}(dp)}{i(2\pi)^d}\what{\phi}(s,\xi-p)\right) ds.
\end{equation}
%where $\what{V}(dp)$ is the spectral measure of $V$,
%\[
%V(x)= \int_{\bbR^d} \frac{\what{V}(dp)}{(2\pi)^d}e^{ip\cdot x}.
%\]
Then \eqref{e.duhamel} is iterated to produce an infinite series
expansion of $\what{\phi}(t,\xi)$. Evaluating the average wave function $\E[\what{\phi}(t,\xi)]$,
or the energy $\E[|\what{\phi}(t,\xi)|^2]$ leads to the  Feynman diagrams
arising from computing the high order moments of the form
$\E[\what{V}(dp_1)\ldots\what{V}(dp_N)]$ for arbitrarily large~$N$. To
pass to the limit requires either delicate oscillatory phase estimates
or some specific structure of the power spectrum so that explicit
calculations can be carried out.  It is unclear whether the diagram expansion can be applied in $d=2$ when we need the renormalization.

We use a different approach in this paper, similar to the one applied to the parabolic setting in \cite{gu2017moments}. For the heat equation with a random potential
\begin{equation}
\label{pam}
\partial_t u_\eps=\frac12\Delta u_\eps+(V_\eps-C_\eps)u_\eps, \  \ u_\eps(0,x)=u_0(x), 
\end{equation}
the Feynman-Kac formula implies 
\begin{eqnarray}\label{e.fkpam}
&&\E[u_\eps(t,x)]=\E\E_\B\Big[u_0(x+B_t)\exp\Big\{\int_0^t V_\eps(x+B_s)ds-C_\eps t\Big\}\Big]\\
&&~~~~~~~~~~~~~=\E_\B\Big[u_0(x+B_t)\exp\Big\{\int_0^t \int_0^s R_\eps(B_s-B_u)duds-C_\eps t\Big\}\Big].\nonumber
\end{eqnarray}
with $R_\eps$ the covariance function of $V_\eps(x)$.
%\begin{equation}
%\label{r-eps}
%R_\eps(x):=\frac{1}{\eps^{d}}R\left(\frac{x}{\eps}\right),\quad x\in\bbR^d,
%\end{equation}
%and the second step in \eqref{e.fkpam} used the fact that $\int_0^t V_\eps(x+B_s)ds$ is of Gaussian distribution conditioning on the Brownian path. 
Using \eqref{032204} one can easily show -- see \eqref{X-eps} below,
that, for $d=2$ and the Schoenberg class covariance function $R(\cdot)$
satisfying  condition
\eqref{021904}, we have
\begin{equation}\label{e.congamma}
\lim_{\eps\to0}  \int_0^t \int_0^s(R_\eps(B_s-B_u)-\E_\B[R_\eps(B_s-B_u)])duds=
  \bar R_d\gamma([0,t]_<^2),\quad\mbox{in $L^2$.}
\end{equation}
In this case, the average intersection time in $d=2$ is
 \begin{equation}\label{jun104}
\int_0^t\int_0^s
 \E_\B[R_\eps(B_s-B_u)]duds\sim C_\eps t=\frac{\bar R_2 t}{\pi}\log\eps^{-1}.
 \end{equation}
 In $d=1$, the mean on the left side %$\int_0^t\int_0^s \E_\B[R_\eps(B_s-B_u)]duds$
 converges and no renormalization is needed, so  
 $C_\eps=0$. It was proved in
 \cite{pardoux2006homogenization} for $d=1$ and in
 \cite{hairer2015simple} for $d=2$ that $u_\eps$ converges to the
 solution to a limiting SPDE. By passing to the limit on both sides of
 \eqref{e.fkpam}, a  representation for the moments of~$u_\eps$ can be obtained, see \cite{gu2017moments}.

The idea of the proof of Theorem~\ref{t.mainth} is similar:  \eqref{e.unre} is rewritten as 
\[
\partial_t \phi=\frac{i}{2}\Delta\phi-i\delta V(x)\phi,
\]
and the Feynman-Kac formula can be used to formally express $\phi$ as an average 
with respect to the Brownian motion with an ``imaginary diffusivity'',
written as $\sqrt{i}B_t$. 
Thus, we need to design a Feynman-Kac type 
formula for $\E[\what{\phi}_\eps(t,\xi)]$ similar to \eqref{e.fkpam}, and
prove a parallel version of \eqref{e.congamma} with $R_\eps$  replaced by a
corresponding complex function in the case of the Schr\"odinger equation.  
%Representing the solution to the Schr\"odinger equation in terms of functional integration falls into the classical subject of Feynman path integral, 
%but it seems to be the first time for such  probabilistic approaches to be applied to the random Schr\"odinger equation.

It is natural to ask what happens in
dimensions $d\geq 3$. The approach used here breaks down -- in $d\geq
3$, the renormalized self-intersection local time of Brownian motion
does not exist \cite{calais-yor,yor1985renormalisation} since the variance also blows up. For the parabolic setting in
$d=3$, the mean of 
\[
\int_0^t \int_0^s R_\eps(B_s-B_u)duds
\]
diverges as
$\eps^{-1}$ and its variance diverges as $\log \eps^{-1}$, so two
renormalization constants are needed -- it has been proved in
\cite{hairer2015multiplicative} that with
\[
C_\eps=c_1\eps^{-1}+c_2\log\eps^{-1},
\]
and appropriate $c_1,c_2$, the
solution $u_\eps$ converges to a non-trivial random limit. 
However,~$\E[u_\eps]$ blows up in the limit~\cite{gu2017moments}. %It is unclear to us what happens to the solution of the Schr\"odinger equation.

The rest of the paper is organized as follows. In Section~\ref{s.fk},
we present a Feynman-Kac representation for the average wave function
which corresponds to \eqref{e.fkpam} in the parabolic setting. In
Section~\ref{s.wk}, we prove the convergence to the renormalized
self-intersection local time in \eqref{e.congamma}, where the
Schoenberg class $R_\eps$ is replaced by the respective ``mixture'' of
free Schr\"odinger kernels. The proof relies on an application of the Clark-Ocone formula which is recalled in the appendix. In Section~\ref{s.uni}, 
%we prove some uniform integrability result and 
we pass to the limit in the Feynman-Kac representation. The homogenization result is shown in  Section~\ref{sec5}.

Throughout the paper, we define $\sqrt{i}=(1+i)/\sqrt{2}$, and we use
$a\les b$ to denote $a\leq Cb$ for some constant $C>0$ independent of
$\eps$, and the constants denoted by  $C$ may differ from line to line.

{\bf Acknowledgment.}   We would like to thank the anonymous referees for a very careful reading of the manuscript and many helpful suggestions and comments, in particular for pointing out that the small time constraint we had in the original version of Theorem~\ref{t.homo} can be removed. YG is partially supported by the NSF
grant DMS-1613301, T.K by the NCN grant 2016/21/B/ST1/00033 and LR by the NSF grants DMS-1311903 and DMS-1613603. TK wishes to express his gratitude to
Prof. A. Talarczyk-Noble for valuable discussions during the course of
preparation of the article.

\section{A Feynman-Kac formula for the average  wave function}
\label{s.fk}

In this section, we prove the Feynmann-Kac representation for the average wave function.
We understand the solution of the Schr\"odinger equation 
\begin{equation}\label{e.sfk}
i\partial_t \phi+\frac12\Delta\phi-V(x)\phi=0, \  \ \phi(0,x)=\phi_0(x),
\end{equation}
in terms of the corresponding Duhamel series expansion~\cite{chen2016weak}. A standard argument, as, for instance,
in~\cite[Proposition 2.2 part (iii)]{chen2016weak},
shows that, even though the potential $V(x)$ is unbounded,
 (\ref{e.sfk}) preserves the $L^2(\R^d)$ norm of the solution:
\[
\E\|\what{\phi}(t,\cdot)\|_{L^2(\bbR^d)}^2=\|\what{\phi}_0\|_{L^2(\bbR^d)}^2,
\]
and the function
$\bar\phi(t,\xi):=\E[\what{\phi}(t,\xi)]$ belongs to $L^2(\bbR^d)$ for each $t\ge0$.
 \begin{proposition}\label{p.fk}
The solution of (\ref{e.sfk}) 
satisfies, point-wise in $(t,\xi)$: 
\begin{equation}\label{e.fkre}
\E[\what{\phi}(t,\xi)]=\what{\phi}_0(\xi) \E_\B\left[ \exp\left\{i\sqrt{i}\xi\cdot B_t-\frac12\int_{[0,t]^2} R(\sqrt{i}(B_s-B_u))duds\right\}\right].
\end{equation}
\end{proposition}
To make sense of (\ref{e.fkre}), we may extend the function $R(x)$ 
to the domain $\bar D\subset\mathbb C^d$, where 
\[
D:=\{zx:\,x\in\bbR^d,\, z\in \mathbb D_0\}, 
~~~~ \mathbb D_0:=\{z\in \mathbb{C}: \mathrm{Re}\ z^2>0\},
 \]
by setting $R(zx)=\rho(z|x|)$, with $\rho(r)$ given by \eqref{012502a}. 
Then, $R(\sqrt{i}(B_s-B_u))$ is uniformly bounded for all $s,u\geq 0$ and the r.h.s. of \eqref{e.fkre} is well-defined.

%To see a quick formal reason why (\ref{e.fkre}) should hold, we formally apply the Feynman-Kac formula to \eqref{e.sfk}: 
%\[
%\phi(t,x)=\E_\B\Big[ \phi_0(x+\sqrt{i}B_t) \exp\Big\{-i\int_0^t V(x+\sqrt{i}B_s)ds\Big\}\Big],
%\]
%and take the expectation of both sides
%\begin{eqnarray*}
%&&\E[\phi(t,x)]=\E_B\E\Big[\phi_0(x+\sqrt{i}B_t) \exp\Big\{-i\int_0^t V(x+\sqrt{i}B_s)ds\Big\}\Big]\\
%&&~~~~~~~~~~~~=\E_B[\phi_0(x+\sqrt{i}B_t)\exp\Big\{-\frac12\int_{[0,t]^2} R(\sqrt{i}(B_s-B_u))duds\Big\}\Big].
%\end{eqnarray*}
%Taking the Fourier transform on both sides  leads
%to \eqref{e.fkre}.
% 
We note that another expression for
$\E[\what{\phi}(t,\xi)e^{\frac{i}{2}|\xi|^2t}]$ was obtained in \cite[Proposition 2.1]{chen2016weak}
but it is less suitable for our analysis.

\subsubsection*{Proof of Proposition \ref{p.fk}}

We fix $(t,\xi)$ and define the function
\[
F_1(z):=\E_\B\Big[ \exp\Big\{iz\xi\cdot B_t-\frac12\int_{[0,t]^2}
  R(z(B_s-B_u))duds\Big\}\Big],
\]
as well as the corresponding Taylor expansion
\[
F_2(z)=\sum_{n=0}^\infty F_{2,n}(z),\quad z\in \bar{\mathbb D}_0,
\]
with
\[
F_{2,n}(z) :=\frac{(-1)^n}{2^n(2\pi)^{nd}n!} \int_{[0,t]^{2n}}\int_{\R^{nd}}\prod_{j=1}^n \what{R}(p_j) \E_\B\left[ e^{iz\xi\cdot B_t}  \prod_{j=1}^ne^{izp_j\cdot(B_{s_j}-B_{u_j})}\right] dpdsdu.
\]
 It is straightforward to check that both 
$F_1$ and $F_2$ are analytic on $\mathbb D_0$ and continuous on $\bar{\mathbb
  D}_0$. 
Note that $\sqrt{i}\in\partial \mathbb D_0$.
The goal is to show that 
\begin{equation}\label{may1806}
\E[\what{\phi}(t,\xi)]=\what{\phi}_0(\xi)F_1(\sqrt{i}).
\end{equation}
Since $(z,s,u)\mapsto R(z(B_s-B_u))$ is  bounded on $ \bar{\mathbb D}_0\times\bbR_+^2$, we have 
\begin{equation}
\begin{aligned}
F_1(z)=&\sum_{n=0}^\infty \frac{(-1)^n}{2^n n!} \E_\B\left[ e^{iz\xi\cdot B_t} \left(\int_{[0,t]^{2}}R(z(B_s-B_u))dsdu\right)^n\right]\\
=&\sum_{n=0}^\infty \frac{(-1)^n}{2^n n!} \E_\B\left[ e^{iz\xi\cdot B_t} \int_{[0,t]^{2n}}\prod_{j=1}^n R(z(B_{s_j}-B_{u_j}))dsdu\right]\\
=&\sum_{n=0}^\infty \frac{(-1)^n}{2^n(2\pi)^{nd}n!} \E_\B\left[
  e^{iz\xi\cdot B_t}
  \int_{[0,t]^{2n}}\int_{\R^{nd}}\prod_{j=1}^n \what{R}(p_j)
  e^{izp_j\cdot(B_{s_j}-B_{u_j})}dp dsdu\right],\quad z\in \bar{\mathbb D}_0.
\end{aligned}
\end{equation}
For $z=x\in\R$, we can apply the Fubini theorem to see that
$
F_1(x)=F_2(x).
$
Due to the analyticity and continuity of
$F_1$ and $F_2$, we therefore have $F_1(z)=F_2(z)$ for all $z\in \bar{\mathbb D}_0$. 
Hence,~(\ref{may1806}) is equivalent to  
\begin{equation}
\label{022204}
\E[\what{\phi}(t,\xi)]=\what{\phi}_0(\xi)\sum_{n=0}^\infty
F_{2,n}(\sqrt{i}),
\end{equation}
and this is what we will show. 
For a fixed $n$, we rewrite 
\[
\begin{aligned}
F_{2,n}(\sqrt{i})=\frac{(-1)^n}{2^n(2\pi)^{nd}n!} \int_{[0,t]^{2n}}\int_{\R^{2nd}}&
\prod_{j=1}^n \what{R}(p_{2j-1})\delta(p_{2j-1}+p_{2j})\\
&\times  \E_\B\left[ e^{i\sqrt{i}\xi\cdot B_t} e^{-\sum_{j=1}^{2n} i\sqrt{i}p_jB_{s_j}}\right]dsdp.
\end{aligned}
\]
%where we changed $p\mapsto -p$. 
Let $\sigma$ denote a permutation of
$\{1,\ldots,2n\}$. After a suitable relabeling of the $p$-variables we can write 
\begin{equation}\label{may1802}
\begin{aligned}
F_{2,n}(\sqrt{i}) =\frac{(-1)^n}{2^n(2\pi)^{nd}n!} \sum_{\sigma}\int_{[0,t]_<^{2n}}\int_{\R^{2nd}}&\prod_{j=1}^n \what{R}(p_{\sigma(2j-1)})\delta(p_{\sigma(2j-1)}+p_{\sigma(2j)})\\
&\times  \E_\B\left[ e^{i\sqrt{i}\xi\cdot B_t} e^{-\sum_{j=1}^{2n} i\sqrt{i}p_jB_{s_j}}\right]ds dp,
\end{aligned}
\end{equation}
where $[0,t]_<^{2n}:=\{(s_1,\ldots,s_{2n}): 0\leq s_{2n}\leq\ldots\leq s_{1}\leq t\}$. Let $\F$ denote the pairings formed over $\{1,\ldots,2n\}$. 
It is straightforward to check that 
\begin{equation}\label{may1804}
\begin{aligned}
F_{2,n}(\sqrt{i}) =\frac{1}{i^{2n}(2\pi)^{nd}} \sum_{\F}\int_{[0,t]_<^{2n}}\int_{\R^{2nd}}&\prod_{(k,l)\in \F} \what{R}(p_k)\delta(p_k+p_l)\\
&\times  \E_\B\left[ e^{i\sqrt{i}\xi\cdot B_t} e^{-\sum_{j=1}^{2n} i\sqrt{i}p_jB_{s_j}}\right]ds dp.
\end{aligned}
\end{equation}
%See~\cite{chen2016weak} for a detailed explanation of the mapping between the sets of permutations and   pairings that accounts for the $2^nn!$ difference in the pre-factors in (\ref{may1802}) and (\ref{may1804}). 
The pre-factors in (\ref{may1802}) and (\ref{may1804}) differ by a factor of $2^nn!$ since $i^{-2n}=(-1)^n$, and this comes from the mapping between the sets of permutations and pairings. Briefly speaking, for a given pairing with $n$ pairs, we have $n!$ ways of permutating the pairs, and inside each pair, we have $2$ options which leads to the additional factor of $2^n$. This is explained in detail in the proof of~\cite[Proposition 2.1]{chen2016weak}

The phase factor inside the  integral in (\ref{may1804}) can be computed explicitly:
\begin{equation}\label{may1812}
\E_\B\left[ e^{i\sqrt{i}\xi\cdot B_t} e^{-\sum_{j=1}^{2n} i\sqrt{i}p_jB_{s_j}}\right]=e^{-\frac{i}{2}|\xi|^2(t-s_1)-\frac{i}{2}|\xi-p_1|^2(s_1-s_2)-\ldots-\frac{i}{2}|\xi-\ldots-p_{2n}|^2s_{2n}}.
\end{equation}

On the other hand, using the Duhamel expansion, we can write the solution $\what{\phi}(t,\xi)$  as an infinite series 
\begin{equation}\label{may1810}
\begin{aligned}
\what{\phi}(t,\xi)=\sum_{n=0}^\infty \int_{[0,t]_<^n}\int_{\R^{nd}}& \prod_{j=1}^n \frac{\what{V}(dp_j)}{i(2\pi)^d} e^{-\frac{i}{2}|\xi|^2(t-s_1)-\frac{i}{2}|\xi-p_1|^2(s_1-s_2)-\ldots-\frac{i}{2}|\xi-\ldots-p_{n}|^2s_{n}}\\
&\times \what{\phi}_0(\xi-p_1-\ldots-p_n) ds.
\end{aligned}
\end{equation}
Evaluating the expectation $\E [\what{\phi}(t,\xi)]$ in \eqref{may1810}, using the pairing
formula for computing the Gaussian moment 
\[
\E[\what{V}(dp_1)\ldots\what{V}(dp_n)],
\]
and the  fact that 
\[
\E[\what{V}(dp_i)\what{V}(dp_j)]=(2\pi)^d\what{R}(p_i)\delta(p_i+p_j)dp_idp_j,
\]
and comparing the result to (\ref{may1804})-(\ref{may1812}), 
we conclude that \eqref{022204} holds, completing the proof.
\qed

\section{Convergence to the renormalized self-intersection local time}
\label{s.wk}

By Proposition~\ref{p.fk}, the average of the solution to \eqref{e.maineq1} is written as
\[
\E[\what{\phi}_\eps(t,\xi)]=\what{\phi}_0(\xi)\exp\left\{-iC_\eps t\right\}\E_\B\left[ \exp\left\{i\sqrt{i}\xi\cdot B_t-\int_{[0,t]^2_{<}} R_\eps(\sqrt{i}(B_s-B_u))dsdu\right\}\right],
\]
with 
\begin{equation}\label{r-eps}
R_\eps(x):=\frac{1}{\eps^{d}}R\left(\frac{x}{\eps}\right),\quad x\in\bbR^d.
\end{equation}
Compared with \eqref{e.fkre}, we do not have the $1/2$ factor in the above probabilistic representation since the integration domain of $s,u$ is changed to $[0,t]_<^2$.
%$R_\eps(x)$ given by \eqref{022502a} and \eqref{r-eps}. 
We define 
\[
%\begin{aligned}
X_\eps(t):=\int_{[0,t]^2_{<}} R_\eps(\sqrt{i}(B_s-B_u))dsdu
=  \frac{1}{\eps^{d}}\int_0^{+\infty} \mu(d\lam)\int_{[0,t]^2_{<}} e^{-\frac{i\lam^2}{2\eps^2}|B_s-B_u|^2} dsdu.
%\end{aligned}
\]
The goal of this section is to prove the $L^2$ convergence of
$X_\eps(t)+iC_\eps t$, as $\eps\to0$. Let $q_t(x)$ be the Gaussian
kernel given by \eqref{q-eps}. We denote by 
\[
s_t(x):=q_{it}(x)=\frac{1}{(2\pi
  it)^{d/2}}e^{-\frac{|x|^2}{2it}},\quad t\in\bbR,
\]
the free Schr\"odinger kernel, the solution of 
\[
i\partial_t s_t+\frac12\Delta s_t=0, \    \ s_0(x)=\delta(x),
\]
and also set
\begin{equation}\label{e.appx}
\cX_\tau (t):=\int_{[0,t]^2_{<}} s_\tau (B_s-B_u)ds du.
\end{equation}
It is straightforward to check that 
\begin{equation}
\label{X-eps}
X_\eps(t)=(-2\pi i)^{d/2}\int_0^{+\infty}\cX_{\eps^2\lam^{-2}}^*(t)\frac{\mu(d\lam)}{\lam^{d}}.
\end{equation}
The expectation of the solution to \eqref{e.maineq1} can be written as 
\begin{eqnarray}\label{e.fkre1}
&&\!\!\!\!\!\!\!\!\!
\E[\what{\phi}_\eps(t,\xi)]=\what{\phi}_0(\xi) \exp\left\{-iC_\eps t\right\}\E_\B\left[ \exp\left\{i\sqrt{i}\xi\cdot B_t-X_\eps(t)\right\}\right]\\
&&~~~~~~~~~
=\what{\phi}_0(\xi)\exp\left\{-iC_\eps t\right\} \E_\B\left[
\exp\left\{i\sqrt{i}\xi\cdot B_t-(-2\pi i)^{d/2}\int_0^{+\infty}\cX_{\eps^2\lam^{-2}}^*(t)\lam^{-d}\mu(d\lam)\right\}\right]\nonumber,
\end{eqnarray}
which, in turn, can be split as
\begin{eqnarray}\label{may3102}
&&\E[\what{\phi}_\eps(t,\xi)]=\what{\phi}_0(\xi)\exp\left\{-iC_\eps t\right\} 
\exp\Big\{-(-2\pi i)^{d/2}\int_0^{+\infty}\E_\B[\cX_{\eps^2\lam^{-2}}^*(t)]\lam^{-d}\mu(d\lam)\Big\} \\
&&~~~~~~~~~~~~\times
\E_\B\Big[
  \exp\Big\{i\sqrt{i}\xi\cdot B_t-(-2\pi i)^{d/2}\int_0^{+\infty}[\cX_{\eps^2\lam^{-2}}^*(t)-\E_\B[\cX_{\eps^2\lam^{-2}}^*(t)]]
  \lam^{-d}\mu(d\lam)\Big\}\Big].\nonumber
\end{eqnarray}
We will show that the terms in the first line in (\ref{may3102}) compensate each other, and the term in the second line has a limit.
We begin with the latter. 
\begin{proposition}\label{p.weakcon}
In $d=1,2$,
\begin{equation}\label{052204}
\lim_{\tau\to0}\left\{\cX_\tau(t)-\E_\B[\cX_\tau(t)]\right\}= \gamma([0,t]^2_<),\quad  \mbox{for any $t>0$},
\end{equation}
 in $L^2(\Sigma)$, with $\gamma([0,t]_<^2)$ defined in (\ref{032204}). %Section~\ref{s.deflocal}.
In addition, we have
\begin{equation}\label{042204}
\sup_{\tau>0}\E_{\B}|\cX_\tau (t) -\E_\B[\cX_\tau(t)]|^2<+\infty,\quad  \mbox{for any $t>0$}.
\end{equation}
\end{proposition}

If the free Schr\"odinger kernel in \eqref{e.appx} is
replaced by the heat kernel, Proposition~\ref{p.weakcon} is classical
and reduces to  the convergence expressed in \eqref{032204}. Although, on the formal level, $q_\tau(x)$ and $s_\tau(x)$ both converge to the Dirac function as $\tau\to0$, it is surprising that the oscillation in $s_\tau$ does not change the asymptotic behavior of $\cX_\tau-\E_\B[\cX_\tau]$.

For the analysis of the intersection local time of the Brownian motion
(and more generally, the fractional Brownian motion), the Clark-Ocone formula turns out to be a
convenient tool, see \cite{hu2008integral}.  
For a fixed $\tau>0$, and $t>0$, we let
\begin{equation}
\label{chi}
\chi_{\tau}(t,r):=\int_r^t\left[ \int_0^r \nabla q_{i\tau
    +s-r}(B_r-B_u)du\right]ds,\quad 0\le r\le t.
\end{equation}
%Fix $t>0$, it is clear that 
The process $\left(\chi_\tau(t,r)\right) $, $0\le r\le t$, is adapted
with respect to the natural filtration ${\cal F}_r$ of the
Brownian motion. As we show in the appendix, see (\ref{e.cko}), we have
%By \eqref{e.cko}, we have
\begin{equation}\label{e.made}
\cX_\tau(t)-\E_\B[\cX_\tau(t)]=\int_0^t \chi_{\tau}(t,r)dB_r,
\end{equation}
with the stochastic integral understood in the It\^o sense.
The renormalized self-intersection local time has the stochastic integral representation (see e.g. \cite[Theorem 2]{hu2008integral} for a more general result on the fractional Brownian motions):
\begin{equation}\label{may2902}
\gamma([0,t]_<^2)=\int_0^t \chi_{0}(t,r) dB_r.
\end{equation}
Formally, the convergence of
$\cX_\tau(t)-\E_\B[\cX_\tau(t)]$ towards $\gamma([0,t]_<^2)$, as
$\tau\to0$ follows from the fact that $\lim_{\tau\to0}\chi_{\tau}(t,r)=\chi_{0}(t,r)$.

\subsubsection*{Proof of Proposition~\ref{p.weakcon}}

Let $\cY_\tau(t):=\cX_\tau(t)-\E_\B[\cX_\tau(t)]$ and consider the covariance
\[
\E_\B[\cY_{\tau_1}(t)\cY_{\tau_2}^*(t)]
=\int_0^t\!\left(\int_{[r,t]^2}\!\int_{[0,r]^2} \E_\B[\nabla q_{i\tau_1+s_1-r}(B_r-B_{u_1})
 \cdot \nabla q_{i\tau_2+s_2-r}^*(B_r-B_{u_2})] duds\right) dr.
\]
%Our goal is to show it converges as $\tau_1,\tau_2\to0$. 
We write the expectation inside the integral in the Fourier domain
\[
\begin{aligned}
&\E_\B[\nabla q_{i\tau_1+s_1-r}(B_r-B_{u_1}) \cdot \nabla q_{i\tau_2+s_2-r}^*(B_r-B_{u_2})] \\
&=\frac{1}{(2\pi)^{2d}} \int_{\R^{2d}} \E_\B[e^{i\xi_1\cdot (B_r-B_{u_1})} e^{-i\xi_2\cdot(B_r-B_{u_2})}] 
(\xi_1\cdot \xi_2) e^{-\frac12|\xi_1|^2(i\tau_1+s_1-r)}e^{-\frac12|\xi_2|^2(-i\tau_2+s_2-r)} d\xi,
\end{aligned}
\]
and claim that the non-negative function 
\[
F(\xi,u,s,r):= \E_\B[e^{i\xi_1\cdot (B_r-B_{u_1})} e^{-i\xi_2\cdot(B_r-B_{u_2})}] |\xi_1||\xi_2| e^{-\frac12|\xi_1|^2(s_1-r)}e^{-\frac12|\xi_2|^2(s_2-r)}
\]
satisfies 
\begin{equation}\label{e.inte}
{\cal I}(t):=\int_0^t dr\int_{[r,t]^2}\int_{[0,r]^2}duds\int_{\R^{2d}} F(\xi,u,s,r)d\xi<+\infty.
\end{equation}
Then, by the dominated convergence theorem, we 
deduce that $\E_\B[\cY_{\tau_1}(t)\cY_{\tau_2}^*(t)]$ converges as~$\tau_1,\tau_2\to0$,  hence
$\cY_{\tau}(t)$ is a Cauchy sequence and converges in $L^2(\Sigma)$. The same argument also implies that 
\[
\lim_{\tau\to0}\E_\B |\cY_\tau(t)-\gamma([0,t]_<^2)|^2=0,
\]
because of (\ref{may2902}).

We turn to the proof of \eqref{e.inte}. Fix $t>0$ and note that
 $$
\int_0^t e^{-\lambda s}ds\leq \frac{c(t)}{1+\lambda}
$$
for any $\lambda>0$ with  
 $$
c(t):=\sup_{\lam>0}\frac{1+\lam}{\lam}(1-e^{-\lam t}).
$$ 
Using this estimate, we first integrate in $s$, and then take the expectation, to obtain, with the constant in the
"$\les$" inequality dependent on $t$:
\[
\begin{aligned}
&{\cal I}(t)\les  \int_0^t \int_{[0,r]^2}\int_{\R^{2d}}\E_\B[e^{i\xi_1\cdot (B_r-B_{u_1})} e^{-i\xi_2\cdot(B_r-B_{u_2})}]  \frac{|\xi_1||\xi_2|}{(1+|\xi_1|^2)(1+|\xi_2|^2)} dr dud\xi \\
&=2\int_0^t \int_{[0,r]^2}\int_{\R^{2d}}\E_\B[e^{i\xi_1\cdot (B_r-B_{u_1})} e^{-i\xi_2\cdot(B_r-B_{u_2})}] \frac{1_{\{u_2<u_1\}} |\xi_1||\xi_2|}{(1+|\xi_1|^2)(1+|\xi_2|^2)} dr  du  d\xi \\
&=2\int_0^t \int_{[0,r]^2}\int_{\R^{2d}} e^{-\frac12|\xi_1-\xi_2|^2(r-u_1)} e^{-\frac12|\xi_2|^2(u_1-u_2)} \frac{1_{\{u_2<u_1\}} |\xi_1||\xi_2|}{(1+|\xi_1|^2)(1+|\xi_2|^2)} dr du d\xi .
\end{aligned}
\]
We further integrate in $u$ and $r$ and see that
\begin{equation}\label{072204}
{\cal I}(t)\les \int_{\R^{2d}}\frac{|\xi_1||\xi_2|d\xi_1 d\xi_2}{(1+|\xi_1-\xi_2|^2)(1+|\xi_1|^2)(1+|\xi_2|^2)^2} <+\infty,
\end{equation}
as $d\leq 2$, which is (\ref{e.inte}).
To conclude that \eqref{042204} holds, it suffices to observe  that by virtue of
\eqref{072204} we have
\[
\sup_{\tau>0}\E_\B[|\cY_{\tau}(t)|^2]\les {\cal I}(t)<+\infty,
\]
finishing the proof of Proposition~\ref{p.weakcon}.

\subsubsection*{Re-centering as the compensating constant}

 Going back to (\ref{may3102}), we now show that 
the recentering of the intersection local time $\E_\B[\cX_\tau(t)]$
coincides with the renormalization of the random PDE by the addition of the term $C_\eps$, 
so that the two terms in the first line of (\ref{may3102}) cancel up to 
a $O(1)$ constant. % It is characterized as follows.
%An explicit calculation gives
\begin{lemma}\label{l.reconst}
We have, for each $t>0$ fixed,
\[
\E_\B[\cX_\tau(t)]= \left|\begin{array}{ll}
\dfrac{(2  t)^{\frac32}}{3\sqrt{\pi}}+o(1), & \mbox{when }d=1, \\
&\\
 \dfrac{t}{2\pi}\log\left( \dfrac{t}{e\tau}\right)-\dfrac{it}{4}+o(1),& \mbox{when }d=2,
 \end{array}
 \right.
\]
as $\tau\to0$. In addition, we have $\sup_{\tau>1}|\E_\B[\cX_{\tau}(t)]|\les 1$.
%\[
%\sup_{\tau>0}|r_d(\tau)|\les \left|\begin{array}{ll}
%1 , & \mbox{when } d=1, \\
%1+1_{\{\tau>1\}}\log \tau , & \mbox{when } d=2.
%\end{array}
%\right.
\end{lemma}

\begin{proof}
By a direct calculation, we have 
\begin{equation}\label{062204}
\E_\B[\cX_\tau(t)]=\int_{[0,t]^2_{<}} \E_\B[s_\tau(B_s-B_u)]duds
=\frac{1}{(2\pi i)^{{d}/{2}}} \int_0^t ds\int_0^s \frac{du}{[\tau-i(s-u)]^{{d}/{2}}},
\end{equation}
so it is clear that $\sup_{\tau>1}|\E_\B[\cX_{\tau}(t)]|\les 1$.

Next, when $d=1$, we have 
\[
\E_\B[\cX_\tau(t)]=\frac{1}{\sqrt{2\pi i}} \int_0^t ds\int_0^s \frac{1}{\sqrt{-iu}}du +o(1)
= \frac{4 t^{\frac32}}{3\sqrt{2\pi}}+o(1).
\]
%It is also easy to see that $\sup_{\tau>0}|\E_\B[\cX_\tau(t)]|\les 1$, so the $o(1)$ term is uniformly bounded.
When $d=2$, we have 
\begin{align}
\label{012604}
\E_\B[\cX_\tau(t)]=&\frac{1}{2\pi i} \int_0^t ds\int_0^s \frac{\tau du}{\tau^2+u^2} +\frac{1}{2\pi }\int_0^t ds\int_0^s \frac{udu}{\tau^2+u^2} \nonumber\\
=&\frac{1}{2\pi i} \int_0^t ds\int_0^{s/\tau} \frac{ du}{1+u^2} +\frac{1}{2\pi }\int_0^t ds\int_0^{s/\tau} \frac{udu}{1+u^2}.
%= & \frac{t}{2\pi} \log\tau^{-1}+\frac{t}{2\pi}(\log t-1)-\frac{it}{4}+o(1).
\end{align}
The first integral is uniformly bounded in $\tau>0$ and converges as $\tau\to0$. For the second integral, we have 
\[
\frac{1}{2\pi}\int_0^t ds \int_0^{s/\tau} \frac{udu}{1+u^2}=\frac{1}{4\pi}\int_0^t \log\frac{\tau^2+s^2}{\tau^2} ds.
\]
Passing to the limit $\tau\to 0$ in the integral on the right side completes the proof.
%Evaluating the integral of $\int_0^t \log s^2 ds$ completes the proof.
\end{proof}

\section{Uniform integrability and passing to the limit}
\label{s.uni}

We now pass to the limit in (\ref{may3102}) that we write as
\begin{equation}\label{e.fkre2}
\E[\what{\phi}_\eps(t,\xi)]
=\what{\phi}_0(\xi)\exp\Big\{-(-2\pi
  i)^{d/2}\int_0^{+\infty}\E_{\B}[\cX_{\eps^2\lam^{-2}}^*(t)]\lam^{-d}\mu(d\lam)-iC_\eps
  t\Big\}\E_{\B}[Z_\eps(t,\xi)],
\end{equation}
where
\begin{equation}
\label{Z-eps}
Z_\eps(t,\xi):=
\exp\left\{i\sqrt{i}\xi\cdot B_t-(-2\pi i)^{d/2}\int_0^{+\infty}\{\cX_{\eps^2\lam^{-2}}^*(t)-\E_{\B}[\cX_{\eps^2\lam^{-2}}^*(t)]\}\lam^{-d}\mu(d\lam)\right\}.
\end{equation}
We first prove the convergence of the constant factor.
\begin{lemma}\label{l.rho}
With the $C_\eps$ given in \eqref{e.defrho1} and $\rho_d(t)$ given in \eqref{e.defrho0}, we have 
\begin{equation}\label{e.defrho}
-(-2\pi i)^{d/2}\int_0^{+\infty}\E_{\B}[\cX_{\eps^2\lam^{-2}}^*(t)]\lam^{-d}\mu(d\lam)-iC_\eps t\to \rho_d(t).
\end{equation}
\end{lemma}
\begin{proof}
We fix $t>0$ and apply Lemma~\ref{l.reconst}. In $d=1$, using the fact that 
\[
\lim_{\tau\to0}\E_\B[\cX_\tau]= \dfrac{(2  t)^{\frac32}}{3\sqrt{\pi}}  \mbox{ and } \sup_{\tau>0}|\E_\B[\cX_{\tau}]|\les 1,
\]
we send $\eps\to0$ in \eqref{e.defrho} to obtain the result.

In $d=2$, we write 
\[
\int_0^{+\infty}\E_{\B}[\cX_{\eps^2\lam^{-2}}^*(t)]\lam^{-d}\mu(d\lam)=\left(\int_0^{\eps}+\int_\eps^{+\infty}\right)\E_{\B}[\cX_{\eps^2\lam^{-2}}^*(t)]\lam^{-d}\mu(d\lam).
\]
For the integral over the interval $(0,\eps)$, we have $\eps^2\lambda^{-2}>1$.
As 
\[
\sup_{\tau>1}|\E_\B[\cX_{\tau}]|\les 1,
\]
we conclude the integral goes to zero in the limit. For the integral over $[\eps,+\infty)$, we have the estimate
\[
\left|\E_{\B}[\cX_{\eps^2\lam^{-2}}^*(t)]-\frac{t}{2\pi}\log\left(\frac{t\lambda^2}{e\eps^2}\right)-\frac{it}{4}\right|\les 1
\]
uniformly in $\lambda\geq \eps$, and the left side above goes to zero as $\eps\to0$ for each such fixed $\lambda $. 
Now we only need to note that
\begin{equation}\label{e.prd2}
2\pi i\int_\eps^{+\infty}  \left(\frac{t}{2\pi}\log\left(\frac{t\lambda^2}{e\eps^2}\right)+\frac{it}{4}\right)\lambda^{-d}\mu(d\lambda)-iC_\eps t \to \rho_2(t)
\end{equation}
to complete the proof.
\end{proof}

Assumption \eqref{e.r2} is used in \eqref{e.prd2} to pass to the limit. For the above integral in $\lambda$ to be finite, 
we only need 
\[
\int_1^{+\infty} \lambda^{-d}(\log\lambda )\mu(d\lambda)<+\infty.
\]
If 
\[
\int_0^1 \lambda^{-d}|\log \lambda|\mu(d\lambda)=+\infty,
\]
we only need to change $C_\eps$ to remove also  the divergent integral 
\[
\int_{\eps}^{+\infty}\lambda^{-d}(\log\lambda) \mu(d\lambda).
\]

\subsubsection*{The uniform integrability of $Z_\eps(t,\xi)$}

By Proposition~\ref{p.weakcon}, we have 
\begin{equation}\label{e.fkre3}
\begin{aligned}
&Z_\eps(t,\xi)\to Z_0(t,\xi):=\exp\left\{i\sqrt{i}\xi\cdot
  B_t-i^\frac{3d}{2}\bar R_d\gamma([0,t]_<^2)\right\},\quad \mbox{as }\eps\to0,
\end{aligned}
\end{equation}
in probability. To pass to the limit of $\E_\B[Z_\eps(t,\xi)]$ in
\eqref{e.fkre2}, it suffices to show the uniform integrability
of the random variables $Z_\eps(t,\xi)$. For a fixed $t>0$, define
the processes
\begin{align*}
&
M^\tau(s;t):=\int_0^s\chi_{\tau}(t,r)dB_r,\quad \tau>0,
\\
&
N^\eps(s;t):=
  \int_0^{+\infty}M^{\eps^2\lam^{-2}}(s;t)\frac{\mu(d\lam)}{\lam^{d}},\quad \eps>0,\quad0\leq s\leq t,
\end{align*}
where $\chi_\tau(t,r)$ is given by \eqref{chi}. Then, $Z_\eps(t,\xi)$ can be rewritten as 
\begin{equation}\label{e.newz}
Z_\eps(t,\xi)=\exp\left\{i\sqrt{i}\xi\cdot B_t-(-2\pi i)^{d/2}(N^\eps)^*(t;t)\right\}.
\end{equation}
Note that for  fixed $t,\tau,\eps>0$, the processes $\left(M^\tau(s;t)
\right)_{s\in[0,t]}$ and $\left(N^\eps(s;t)
\right)_{s\in[0,t]}$ are continuous trajectory, square integrable, complex-valued martingales.  
Their respective quadratic variations are
\begin{align}
\label{012504}
&
\la M^\tau(\cdot;t)\ra_s=\int_0^s|\chi_{\tau}(t,r)|^2dr,\quad\tau>0,
\\
&
\la N^\eps(\cdot;t)\ra_s=
 \int_0^s\left| \int_0^{+\infty}\chi_{\eps^2\lam^{-2}}(t,r)\frac{\mu(d\lam)}{\lam^{d}}\right|^2dr,\quad \eps>0,\quad 0\leq s\leq t.\nonumber
\end{align}
In other words, $|M^\tau(s;t)|^2- \la M^\tau(\cdot;t)\ra_s$ and $|N^\eps(s;t)|^2-\la N^\eps(\cdot;t)\ra_s$ are local martingales. Using the Cauchy-Schwarz inequality and (\ref{021904}), followed by Jensen's inequality, we conclude that
\begin{align}
\label{022404}
&\E_{\B}\left[\exp\left\{\theta\la N^\eps(\cdot;t)\ra_s\right\}\right]\le
 \E_{\B}\left[\exp\left\{\theta \bar R_d(2\pi)^{-d/2}\int_0^{+\infty}\la
  M^{\eps^2\lam^{-2}}(\cdot;t)\ra_s
  \frac{\mu(d\lam)}{\lam^{d}}\right\}\right]\nonumber\\
&
~~~~~~~~~~~~~~~~~~~~~~~~~~~~~
\le \frac{(2\pi)^{d/2}}{\bar R_d}\int_0^{+\infty}\frac{\mu(d\lam)}{\lam^{d}}\E_{\B}\left[\exp\left\{\theta \bar R_d^2(2\pi)^{-d}\la
  M^{\eps^2\lam^{-2}}(\cdot;t)\ra_s
  \right\}\right],
\end{align}
for any $\theta>0$.
We have the following result:
\begin{proposition}\label{p.uniformin}
For any $\theta>0$, there exists $t_0>0$ such that  
\begin{equation}
\label{012404}
\sup_{t\in[0,t_0]}\sup_{\eps>0}\E_{\B}\left[\exp\left\{\theta\la N^\eps(\cdot;t)\ra_{t}\right\}\right]<+\infty.
\end{equation}
\end{proposition}

\begin{proof} Thanks to \eqref{022404}, the estimate
 \eqref{012404} is a result of  the following claim: for any
$\theta>0$ there exists $t_0>0$ such that
\begin{equation}
\label{012404a}
\sup_{t\in[0,t_0]}\sup_{\tau>0}\E_\B\left[ \exp\left\{\theta \la M^\tau(\cdot;t)\ra_{t}\right\}\right] <+\infty.
\end{equation}
Let us recall   \eqref{chi}:
\begin{equation}\label{may3104}
\chi_{\tau}(t,r)=\frac{1}{(2\pi)^\frac{d}{2}}\int_0^rdu\left\{\int_0^{t-r} \frac{B_u-B_r}{
  (i\tau+s)^{\frac{d}{2}+1}}
e^{-\frac{|B_r-B_u|^2}{2(i\tau+s)}}ds\right\}.
\end{equation}

\subsubsection*{The case $d=1$} 

We shall need the following.
\begin{lemma}\label{l.inte}
There exists a constant $C>0$ such that for all $t,\lambda,\tau>0$, we have 
\[
\left|\int_0^t \frac{1}{(i\tau+s)^{\frac{3}{2}}}e^{-\frac{\lambda}{i\tau+s}} ds \right|\leq \frac{C}{\sqrt{\lambda}}.
\]
\end{lemma}
\begin{proof}
We have 
\[
\begin{aligned}
\left|\int_0^t \frac{1}{(i\tau+s)^{\frac{3}{2}}}e^{-\frac{\lambda}{i\tau+s}} ds\right|\leq& \int_0^t \frac{1}{(\tau^2+s^2)^\frac34} e^{-\frac{\lambda s}{\tau^2+s^2}} ds\\
\les & \left(\int_0^\tau+\int_\tau^t\right)\frac{1}{(\tau+s)^{\frac32}} \exp\left\{-\frac{\lambda}{s+\tau^2/s}\right\}ds:=I_1+I_2.
\end{aligned}
\]
%where $I_1$ and $I_2$ correspond to the first and the second integral.
When $s\in (0,\tau)$, we have $s+\tau^2/s\leq 2\tau^2/s$, so 
\[
\begin{aligned}
I_1\leq & \int_0^\tau \frac{1}{(\tau+s)^\frac32} e^{-\frac{\lambda s}{2\tau^2}} ds=\frac{1}{\sqrt{\tau}} \int_0^1 \frac{1}{(1+s)^\frac32} e^{-\frac{\lambda s}{2\tau}} ds
%\\
%\leq & \frac{1}{\sqrt{\tau}} \int_0^{1}e^{-\frac{\lambda s}{2\tau}} ds
\le \frac{2\sqrt{\tau}}{\lambda}(1-e^{-\frac{\lambda}{2\tau}}) \les \frac{1}{\sqrt{\lambda}}.
\end{aligned}
\]
When $s\in (\tau,t)$, we have $s+\tau^2/s\leq 2s$, so 
\[
\begin{aligned}
I_2 \leq \int_\tau^t  \frac{1}{(\tau+s)^\frac32} e^{-\frac{\lambda}{2s}} ds \leq \int_0^\infty  \frac{1}{s^\frac32} e^{-\frac{\lambda}{2s}}ds 
\les  \frac{1}{\sqrt{\lambda}}.
\end{aligned}
\]
The proof is complete.
\end{proof}

Using the above lemma we conclude that there exists a constant $C>0$ such that 
\[
\left|\int_0^{t-r} \frac{B_r-B_u}{(i\tau+s)^{\frac32}} e^{-\frac{|B_r-B_u|^2}{2(i\tau+s)}} ds\right| \leq C
\]
for all $r\in (0,t)$, which, in light of  \eqref{012504} and (\ref{may3104}), implies 
\[
\la M^\tau(\cdot;t)\ra_t \les\int_0^t r^2 dr =\frac{t^3}{3}.
\]
Thus, \eqref{012404} holds  for all $t_0>0$ in $d=1$, and we can remove the small time constraint in Theorem~\ref{t.mainth} in $d=1$.

\subsubsection*{The case $d=2$}  Integrating out the $s$ variable gives
\[
\chi_{\tau}(t,r)=-\int_0^r \frac{(B_r-B_u)}{\pi|B_r-B_u|^2}\Big(  
e^{-\frac{|B_r-B_u|^2}{2(i\tau +t-r)}}-e^{-\frac{|B_r-B_u|^2}{2i\tau}}\Big) du,
\]
which, together with \eqref{012504}, implies that
there exists $C>0$ such that
\[
\la M^\tau (\cdot;t)\ra_t \le C \int_0^t \left(\int_0^r |B_r-B_u|^{-1}du \right)^2 dr.
\]
Therefore, by the above and  Jensen's inequality, we conclude that
\begin{align}
\label{022504}
\E_\B[\exp\left\{\theta \la M^\tau(\cdot;t)\ra_t\right\}]  \leq &
\E_\B\Big[ \exp\Big\{\theta C\int_0^t \Big(\int_0^r |B_r-B_u|^{-1}du \Big)^2 dr\Big\}\Big]\\
\leq & \frac{1}{t} \int_0^t \E_\B\Big[\exp\Big\{\theta C
    t\Big(\int_0^r |B_r-B_u|^{-1}du \Big)^2\Big\}\Big]dr.\nonumber
\end{align}
Note that, for a fixed $r>0$, we have
\[
\int_0^r |B_r-B_u|^{-1} du \stackrel{\text{law}}{=}\int_0^r \mathscr{R}_u^{-1} du ,
\]
where $\mathscr{R}_u:=|B_u|$, $u\ge0$ is a Bessel process of dimension
$2$. 
An application of the It\^o formula shows that $\left(\mathscr{R}_r\right)_{r\ge0}$ satisfies 
\[
\int_0^r \mathscr{R}_u^{-1} du = 2(\mathscr{R}_r-b_r),\quad r\ge0.
\]
Here, $(b_r)_{r\ge0}$ is a
standard one dimensional Brownian motion. 
Having this in mind, we estimate the utmost right hand side of
\eqref{022504} using the Cauchy-Schwarz inequality and obtain
\begin{equation}\label{e.dec72}
 \E_\B[\exp\left\{\theta \la M^\tau(\cdot;t)\ra_t\right\}]  \leq\frac{1}{t}\int_0^t \Big\{ \E_\B\Big[ \exp\Big\{\theta C t \mathscr{R}_r^2\Big\} 
\vphantom{\int_0^1}\Big]\Big\}^{1/2}\Big\{\E_\B\Big[ \exp\Big\{\theta C t b_r^2\Big\}\vphantom{\int_0^1}\Big]\Big\}^{1/2} dr.
\end{equation}
It is clear that when $t$ is sufficiently small, the last expression
is bounded independent of $\tau$, which completes the proof of
\eqref{012404a}, and thus that of Proposition~\ref{p.uniformin}.
\end{proof}

\subsection*{Proof of Theorem~\ref{t.mainth}}

Now we can finish the proof of the main result. 
By \eqref{e.fkre2},  \eqref{e.defrho} and \eqref{e.fkre3}, it remains to prove the uniform integrability of  random variables 
$Z_\eps(t,\xi)$ given in \eqref{e.newz}. To do so, we bound their second moments. Using the Cauchy-Schwarz inequality
we get
\[
\begin{aligned}
\E_\B[|Z_\eps(t,\xi)|^2]=&\E_\B\Big[ e^{-\sqrt{2}\xi\cdot B_t}  
\exp\Big\{\vphantom{\int_0^1}-2(2\pi)^{d/2}\mathrm{Re}[(-i)^\frac{d}{2}(N^{\eps})^*(t;t)]\Big\}\Big]\\
\leq & \left\{\E_\B\big[e^{-2\sqrt{2}\xi\cdot B_t}\big]\right\}^{1/2}
\Big\{ \E_\B\Big[\exp\Big\{\vphantom{\int_0^1}-4(2\pi)^{d/2}\mathrm{Re}[(-i)^\frac{d}{2}(N^{\eps})^*(t;t)]\Big\}\Big]\Big\}^{1/2}.
\end{aligned}
\]
We wish to show that there exists $t_0>0$ such that the right 
side of the above estimate is uniformly bounded in $\eps>0$ for
$t\in(0,t_0)$. This will obviously imply the uniform integrability 
of~$Z_\eps(t,\xi)$ and complete the proof of Theorem \ref{t.mainth}   after passing to the limit in \eqref{e.fkre2}.
To obtain the desired bound we consider the following martingale for a fixed $t>0$:
\[
{\bf
  N}^\eps(s;t):=-4(2\pi)^{d/2}\mathrm{Re}[(-i)^\frac{d}{2}(N^{\eps})^*(s;t)],\quad
0\le s\le t.
\]
We have
\[
 \la {\bf N}^\eps(\cdot;t)\ra_t  \leq 16(2\pi)^d   \la N^{\eps}(\cdot;t)\ra_t,
\]
for any $\theta>0$. By Proposition~\ref{p.uniformin}, there exists $t_0>0$ depending on $\theta$ such that  
\[
\sup_{t\in[0,t_0]}\sup_{\eps>0} \E_\B[e^{ \theta \la {\bf N}^\eps(\cdot;t)\ra_{t}}] \leq \sup_{t\in[0,t_0]}
\sup_{\eps>0} \E_\B[ e^{16(2\pi)^d\theta \la N^\eps(\cdot;t)\ra_{t}}]<+\infty.
\]
For $\theta=2$, we adjust the respective $t_0$ as in the statement of
Proposition~\ref{p.uniformin}. We have then 
\begin{eqnarray*}
&&
\sup_{t\in[0,t_0]}\sup_{\eps>0}\E_\B[e^{{\bf N}^\eps(t;t)}]=\sup_{t\in[0,t_0]}\sup_{\eps>0}\E_\B\Big[ \exp\Big\{{\bf
      N}^\eps(t;t)-\la {\bf N}^\eps(\cdot;t) \ra_t\Big\}
  e^{\la {\bf N}^\eps(\cdot;t) \ra_t}\vphantom{\int_0^1}\Big]
 \\
 &&
 \leq \sup_{t\in[0,t_0]}\sup_{\eps>0}\Big\{\E_\B\Big[ \exp\Big\{2{\bf
       N}^\eps(t;t)-2\la {\bf N}^\eps(\cdot;t)
   \ra_t\Big\}\vphantom{\int_0^1}\Big]\Big\}^{1/2}
   \Big\{\E_\B
   \Big[\vphantom{\int_0^1}e^{2\la {\bf N}^\eps(\cdot;t)   \ra_t}\Big]\Big\}^{1/2}
\\
 &&
= \sup_{t\in[0,t_0]}\sup_{\eps>0}\Big\{\E_\B
   \Big[\vphantom{\int_0^1}e^{2\la {\bf N}^\eps(\cdot;t)
   \ra_t}\Big]\Big\}^{1/2}<+\infty.
\end{eqnarray*}
In the last line of the above display, we used the fact that $\exp\{2{\bf
       N}^\eps(t;t)-2\la {\bf N}^\eps(\cdot;t)
   \ra_t\}$ is a martingale for fixed $\eps>0$, which comes from the Novikov's condition and the boundedness of $\la {\bf N}^\eps(\cdot;t)
   \ra_t$. The proof of Theorem \ref{t.mainth} is  complete.

\section{Proof of the homogenization result}\label{sec5}

We now prove Theorem~\ref{t.homo}. 
%Recall that $\phi_\eps$ solves \eqref{012104}.
Assume without loss of generality that the initial
condition $\what{\phi}_0(\xi)$ for \eqref{012104} is compactly supported. 
For an arbitrary~$\what{\phi}_0\in L^2(\bbR^2)$, we can argue by an approximation, since  both \eqref{012104} and \eqref{012104h} preserve the~$L^2(\bbR^2)$ norm.

By Proposition~\ref{p.fk} and (\ref{e.fkre1}), we have for any $(t,\xi)$,
\begin{equation}\label{e.dec71}
\begin{aligned}
\E[\what{\phi}_\eps(t,\xi)]=&\what{\phi}_0(\xi) \E_\B\left[ \exp \left\{ i\sqrt{i}\xi\cdot B_t-\frac{1}{\log\eps^{-1}}\int_{[0,t]^2_{<}} R_\eps(\sqrt{i}(B_s-B_u))dsdu\right\}\right]\\
=&\what{\phi}_0(\xi) \E_\B\left[ \exp \left\{ i\sqrt{i}\xi\cdot B_t+\frac{2\pi i}{\log\eps^{-1}} \int_0^{+\infty}\cX_{\eps^2\lam^{-2}}^*(t)\lam^{-2}\mu(d\lam)\right\}\right].
\end{aligned}
\end{equation}
By Lemma~\ref{l.rho} and \eqref{e.defrho1}, we have
\[
\lim_{\eps\to0}\frac{2\pi i}{\log\eps^{-1}}\int_0^{+\infty} \E_\B[\cX_{\eps^2\lambda^{-2}}^*(t)]\lambda^{-2}\mu(d\lambda)=\frac{it\bar{R}_2}{\pi}.
\]
Combining with Proposition~\ref{p.weakcon}, we further derive
\[
\lim_{\eps\to0}\frac{2\pi i}{\log\eps^{-1}} \int_0^{+\infty}\cX_{\eps^2\lam^{-2}}^*(t)\lam^{-2}\mu(d\lam)=  \frac{it\bar{R}_2}{\pi}\quad\mbox{in $L^2(\Sigma)$. }
\]

We claim that for any $t,\theta>0$, there exists $\eps_0>0$ such that 
\begin{equation}\label{e.dec81}
\sup_{\eps\in(0,\eps_0)} \E_\B\left[\exp\left\{\frac{\theta}{|\log \eps|^2} \la N^\eps(\cdot; t)\ra_t\right\}\right]<+\infty.
\end{equation}
This comes from the same proof of Proposition~\ref{p.uniformin} -- we only need to replace $\theta\mapsto \theta/|\log\eps|^2$ and note that the r.h.s. of \eqref{e.dec72} is uniformly bounded in $\eps\in(0,\eps_0)$ for some small $\eps_0$. Thus, by following the proof of Theorem~\ref{t.mainth}, we have 
\[
\begin{aligned}
\lim_{\eps\to0}\E[\what{\phi}_\eps(t,\xi)]&= \what{\phi}_0(\xi) \E_\B\Big[
  \exp\Big\{i\sqrt{i}\xi\cdot B_t+it\frac{\bar R_2}{\pi}\Big\}\Big]
= \what{\phi}_0(\xi) \exp\Big\{-it\Big(\frac{|\xi|^2}{2}-\frac{\bar
    R_2}{\pi}\Big)\Big\}\\
    &=\what{\phi}_{\hom}(t,\xi),
\end{aligned}
\]
for any $t>0,\xi\in\R^2$.

In addition, by applying the Cauchy-Schwarz inequality to \eqref{e.dec71} and using \eqref{e.dec81}, we have the simple estimate 
\[
|\E[\what{\phi}_\eps(t,\xi)]| \les
|\what{\phi}_0(\xi)| \sqrt{\E_\B\left[\left|\exp\{i\sqrt{i}\xi\cdot B_t\}\right|^2\right]} \les|\what{\phi}_0(\xi)| e^{C|\xi|^2t}.
\]
As
$\what{\phi}_0$ has compact support, we have
\[
|\E[\what{\phi}_\eps(t,\xi)]\what{\phi}_{\hom}^*(t,\xi)|\les |\what{\phi}_0(\xi)|^2e^{C|\xi|^2t} \in L^1(\R^2),
\]
thus, by the dominated convergence theorem and the mass conversation 
\[
\E\|\what{\phi}_\eps(t,\cdot)\|_{L^2(\bbR^d)}^2=\|\what{\phi}_0\|_{L^2(\bbR^d)}^2,
\]
we have 
\[
\begin{aligned}
&\int_{\R^2}
\E[|\what{\phi}_\eps(t,\xi)-\what{\phi}_{\hom}(t,\xi)|^2]d\xi
=2\int_{\R^2} |\what{\phi}_0(\xi)|^2 d\xi-2\mathrm{Re}\left[\int_{\R^2} \E[\what{\phi}_\eps(t,\xi)]\what{\phi}_{\hom}^*(t,\xi)d\xi\right]
\to  0,%\quad \mbox{as $\eps\to0$.}
\end{aligned}
\]
as $\eps\to 0$. The proof of Theorem~\ref{t.homo} is complete.\quad

\appendix

\section{The Clark-Ocone formula}
\label{s.co}

We  recall some facts from  the Malliavin calculus for a standard $d$-dimensional Brownian motion $B_r=(B_r^1,\ldots,B_r^d), r\geq 0$ on $(\Sigma,\mathcal{F},\Pb_\B)$ that are used in our argument. We refer to \cite{nualart2006malliavin} for a more detailed presentation.
Let $H=L^2([0,\infty),\R^d)$ be the Hilbert space corresponding to the
standard inner product $\langle\cdot,\cdot\rangle_H$. We define
a mapping $B:H\to L^2(\Sigma,\mathcal{F},\Pb_\B)$ by letting
\[
B(h)=\int_0^\infty h(r)dB_r=\sum_{j=1}^d  \int_0^\infty
h^j(r)dB_r^j,\quad h=(h^1,\ldots,h^d)\in H,
\]
so that
$$
\E_{\B}[B(h_1)B(h_2)]=\langle h_1,h_2\rangle_H\quad\mbox{ for
}h_1,h_2\in H.
$$
For a random variable of the form $F=f(B(h_1),\ldots,B(h_n))$, where
%$n\in \mathbb{Z}_+$, 
$f:\R^n\to \R$ is a smooth function of polynomial
growth and $h_k \in H, k=1,\ldots,n$, the derivative operator is defined as 
\[
D_r^j F=\sum_{k=1}^n \partial_{x_k} f(B(h_1),\ldots,B(h_n)) h^j_k(r),\quad j=1,\ldots,d
\]
and we write $D_r=(D_r^1,\ldots,D_r^d)$. The derivative is a closeable
operator on $L^2(\Sigma)$ with values in $L^2(\Sigma;H)$. Denote by $\mathbb{D}^{1,2}$  the Hilbert space defined as the completion of the random variables $F$ with respect to the product
\[
\la F,G\ra_{1,2}:=\E_\B[FG]+\E_\B\left[\sum_{j=1}^d  \int_0^\infty (D_r^jF)( D_r^j G) dr\right].
\]
The Clark-Ocone formula, see 
\cite[Proposition 1.3.14 p. 46]{nualart2006malliavin}, says that if $F\in \mathbb{D}^{1,2}$, then  
\[
F=\E_\B[F]+\int_0^\infty \E_\B[D_rF|\mathcal{F}_r] dB_r
=\E_\B[F]+\sum_{j=1}^d \int_0^\infty \E_\B[ D_r^j F | \mathcal{F}_r] dB_r^j,
\]
with $\left(\mathcal{F}_r\right)$ the natural filtration corresponding the Brownian. In our case, 
with $F=\cX_\tau(t)$ for fixed $t,\tau>0$, we have 
\[
\cX_\tau(t)=\E_\B[\cX_\tau(t)]+\int_0^\infty \E_\B[ D_r \cX_\tau(t)|\mathcal{F}_r] dB_r.
\]
Recall that 
\[
\cX_\tau(t)=\int_0^t \int_0^s s_\tau(B_s-B_u)duds,
\]
therefore,
\[
D_r\cX_\tau(t)=\int_0^t\int_0^s \nabla s_\tau(B_s-B_u)1_{[u,s]}(r)duds=1_{[0,t]}(r)\int_r^t \int_0^r \nabla s_\tau(B_s-B_u)duds.
\]
This implies
\begin{eqnarray*}
&&\E_\B[D_r\cX_\tau(t)|\mathcal{F}_r]=1_{[0,t]}(r)\int_r^t \int_0^r \E_\B[\nabla s_\tau(B_s-B_u) |\mathcal{F}_r] duds\\
&&=1_{[0,t]}(r)\int_r^t \int_0^r \nabla s_\tau\star q_{s-r}(B_r-B_u)duds
= 1_{[0,t]}(r) \int_r^t \int_0^r \nabla q_{i\tau +s-r}(B_r-B_u)duds.
\end{eqnarray*}
Here, we have used the fact that $s_{\tau}=q_{i\tau}$ and $q_{i\tau}\star q_{s-r}=q_{i\tau+s-r}$. Thus, we have
\begin{equation}\label{e.cko}
\cX_\tau(t)-\E_\B[\cX_\tau(t)]=\int_0^t \left(\int_r^t \int_0^r \nabla q_{i\tau +s-r}(B_r-B_u)duds\right)dB_r,
\end{equation}
which is (\ref{e.made}).

%\bibliographystyle{siam}
%\bibliography{Refs}
\end{document}